\documentclass[10pt, letterpaper]{article}
\usepackage{comment}
\usepackage{graphicx}
\usepackage{amsmath}
\usepackage{amssymb}
\usepackage{color}
\usepackage{multicol}
\usepackage{amsthm}
\usepackage{hyperref}
\usepackage{textcomp}
\newtheorem{conj}{Conjecture}
\newtheorem{remark}{Remark}
\newtheorem{theorem}{Theorem}
\newtheorem{proposition}{Proposition}

\begin{document}
\title{An Alexandrov--Fenchel-type inequality for hypersurfaces in the sphere}
\author{Frederico Gir\~ao \and Neilha M. Pinheiro}
\date{}
\maketitle
\begin{abstract}
We find a monotone quantity along the inverse mean curvature flow and use it to prove an Alexandrov--Fenchel-type inequality for strictly convex hypersurfaces in the $n$-dimensional sphere, $n \geq 3$.
\end{abstract}
\section{Introduction}
\label{intro}
If $\Sigma \subset \mathbb{R}^n$, $n \geq 3$, is a convex hypersurface, then the Alexandrov--Fenchel inequalities (\cite{alexandrov1}, \cite{alexandrov2}) say that
\begin{equation} \label{AF_Rn}
\int_{\Sigma} \sigma_k (\lambda) d\Sigma \geq C_{n,k} \left( \int_{\Sigma} \sigma_{k-1}(\lambda) d\Sigma \right)^{\frac{n-k-1}{n-k}},
\end{equation}
where $\sigma_k(\lambda)$, $ 1 \leq k \leq n-1$, is the $k^{\mathrm{th}}$ elementary symmetric function of the principal curvature vector $\lambda = (\lambda_1, \ldots, \lambda_{n-1})$ of $\Sigma$ and $C_{n,k} > 0$ is a universal constant. Moreover, the equality holds if and only if $\Sigma$ is a round sphere. In \cite{guan-li}, using a certain inverse curvature flow, Guan and Li showed that (\ref{AF_Rn}) holds for any $\Sigma$ which is star-shaped and $k$-convex (which means that $\sigma_i(\lambda) \geq 0$ for $i = 0, 1, \ldots, k$).

The $k=1$ case of (\ref{AF_Rn}), namely
$$ \int_{\Sigma} H  d\Sigma \geq (n-1) \omega_{n-1} \left( \frac{|\Sigma|}{\omega_{n-1}} \right)^{\frac{n-2}{n-1}},$$
where $|\Sigma|$ is the area of $\Sigma$, $H = \sigma_1(\lambda)$ is its mean curvature, and $\omega_{n-1}$ is the area of the unit sphere $\mathbb{S}^{n-1} \subset \mathbb{R}^n$, is a key step in the proof of the Penrose inequality for graphs, given by Lam in \cite{lam} (see also \cite{dLG1} and \cite{mirandola-vitorio}). The general case, in its turn, was used in a crucial way to establish, for graphical manifolds, versions of the Penrose inequality in the context of the so called Gauss--Bonnet--Chern mass \cite{GWW} (see also \cite{li-wei-xiong} and \cite{girao-mota}).

Instead of $\mathbb{R}^n$, let's consider the hyperbolic $n$-space $\mathbb{H}^n$ to be the ambient space. We will use the model of $\mathbb{H}^n$ given by $\mathbb{R}_+ \times \mathbb{S}^{n-1}$ endowed with the metric
$$ dr^2 + (\sinh^2r)h,$$
where $h$ is the round metric on the unit sphere $\mathbb{S}^{n-1}$.
We say that a closed, embedded hypersurface $\Sigma \subset \mathbb{H}^n$ is {\it star-shaped} if it can be written as a radial graph over a geodesic sphere centered at the origin. Also, $\Sigma$ is said to be {\it strictly mean-convex} if its mean curvature $H$ is positive everywhere.

Let $\rho: \mathbb{H} \to \mathbb{R}$ be the function given by
$$ \rho (r) = \cosh r.$$
In \cite{dLG2}, de Lima and the first named author showed the following Alexandrov--Fenchel-type inequality: if $\Sigma \subset \mathbb{H}^n$, $n \geq 3$, is star-shaped and strictly mean-convex, then
\begin{equation} \label{AF_Hn}
\int_{\Sigma} \rho H d\Sigma \geq (n-1)\omega_{n-1} \left[ \left( \frac{|\Sigma|}{\omega_{n-1}} \right)^{\frac{n-2}{n-1}} + 
\left( \frac{|\Sigma|}{\omega_{n-1}} \right)^{\frac{n}{n-1}} \right],
\end{equation}
with the equality holding if and only if $\Sigma$ is a geodesic sphere centered at the origin. The proof uses, among other ingredients, two monotone quantities along the inverse mean curvature flow (IMCF) and an inequality due to Brendle, Hung and Wang \cite{BHW}. Inequality (\ref{AF_Hn}) was conjecture by Dahl, Gicquaud and Sakovich in \cite{DGS}, where they considered the Penrose inequality for hyperbolic graphs. They found a formula for the mass of a graph and (\ref{AF_Hn}) was the only thing left to show in order to establish the Penrose inequality in this context. Similar inequalities, with the mean curvature replaced by the $k$-mean curvature $\sigma_k$, for odd $k$,  were obtained in \cite{GWW2}, where they also proved versions of the Penrose inequality for graphs in the context of the hyperbolic Gauss--Bonnet--Chern mass.

Consider now the unit $n$-sphere $\mathbb{S}^n$, $n \geq 3$, to be the ambient space. As before, denote by $h$ the round metric on the unit sphere $\mathbb{S}^{n-1}$.  Recall that $(0,\pi) \times \mathbb{S}^{n-1}$ endowed with the metric
$$ dr^2 + (\sin^2 r)h,$$
gives a model for the round metric on $\mathbb{S}^n$. In analogy with the case where the ambient is $\mathbb{H}^n$, we consider the function $\rho: \mathbb{S}^n \to \mathbb{R}$ given by
\begin{equation} \label{defn_rho_sphere}
\rho (r) = \cos r.
\end{equation}

We consider an embedded hypersurface $\Sigma \subset \mathbb{S}^n$, which we always assume to be closed, orientable and connected.

If $\Sigma \subset \mathbb{S}^n$ is a geodesic sphere centered at the origin, then a straightforward computation gives
$$ \int_{\Sigma} \rho H d\Sigma = (n-1)\omega_{n-1}
\left[ \left( \frac{|\Sigma|}{\omega_{n-1}} \right)^{\frac{n-2}{n-1}} - \left( \frac{|\Sigma|}{\omega_{n-1}} \right)^{\frac{n}{n-1}} \right].$$ 
Therefore, it is tempting to conjecture that, if $\Sigma \subset \mathbb{S}^n$ satisfies some reasonable convexity assumption, then
\begin{equation} \label{wrong_AF_Sn}
\int_{\Sigma} \rho H d\Sigma \geq (n-1)\omega_{n-1} \left[ \left( \frac{|\Sigma|}{\omega_{n-1}} \right)^{\frac{n-2}{n-1}} - 
\left( \frac{|\Sigma|}{\omega_{n-1}} \right)^{\frac{n}{n-1}} \right].
\end{equation}
As we will see later (Proposition \ref{result_1}),
if $\Sigma$ is a geodesic sphere not centered at the origin, then 
inequality (\ref{wrong_AF_Sn}) does not hold. We conjecture, however, that inequality (\ref{wrong_AF_Sn}) holds if $\Sigma$ is strictly convex and the origin is taken appropriately (see Conjecture \ref{conj}). 

If $\Sigma$ is diffeomorphic to $\mathbb{S}^{n-1}$ and it is contained in an open hemisphere (for example, if $\Sigma$ is strictly convex \cite{docarmo-warner}), then $\Sigma$ divides $\mathbb{S}^n$ in two regions: a region that properly contains a hemisphere, which we call the {\it outer region}, and a region that is properly contained in a hemisphere, which we call the {\it inner region}.

It is proved in \cite{makowski-scheuer} (see also \cite{gerhardt}) that if $\Sigma$ is strictly convex, then the IMCF converges, in finite time, to an equator $E_{\Sigma}$. This equator determines two hemispheres, with one of them containing $\Sigma$. We associate to each strictly convex $\Sigma \subset \mathbb{S}^n$ a point $x \in \mathbb{S}^n$ in the following way:
we let $x(\Sigma)$ be the center of the hemisphere, determined by $E_\Sigma$, that contains $\Sigma$ (looking at the hemisphere as a geodesic ball). We will refer to the point $x(\Sigma)$ as {\it the point associated to $\Sigma$ via the IMCF}. Notice that, if $\Sigma$ is a geodesic sphere, then $x(\Sigma)$ is its center.

For $x \in \mathbb{S}^n$, we consider the function $\rho_x: \mathbb{S}^n \to \mathbb{R}$ given by
$$ \rho_x(q) = \cos (r_x(q)),$$
where $r_x(q)$ denotes the geodesic distance from $q$ to $x$. Thus, if $x$ is the origin, we recover the function $\rho$ defined by (\ref{defn_rho_sphere}). 
\begin{conj} \label{conj}
Let $\Sigma \subset \mathbb{S}^n$ be a closed, orientable and connected embedded hypersurface. If $\Sigma$ is strictly convex, then
$$ \sup_{x \in \mathbb{S}^n} \int_{\Sigma} \rho_x H d\Sigma \geq (n-1)\omega_{n-1} \left[ \left( \frac{|\Sigma|}{\omega_{n-1}} \right)^{\frac{n-2}{n-1}} - 
\left( \frac{|\Sigma|}{\omega_{n-1}} \right)^{\frac{n}{n-1}} \right],$$
Moreover, the equality holds if and only if $\Sigma$ is a geodesic sphere.
\end{conj}
Unfortunately, we couldn't find a proof to Conjecture \ref{conj}. We were able, however, to prove a related inequality. Before we state it, let's introduce some notation.

We define the quantity $\mathcal{K}$ by
$$ \mathcal{K}(\Sigma) = \omega_{n-1} \left( \frac{|\Sigma|}{\omega_{n-1}} \right)^{\frac{n}{n-1}}.$$
If $x \in \mathbb{S}^n$ and $\Sigma$ is diffeomorphic to $\mathbb{S}^{n-1}$ and contained in an open hemisphere, we define the quantity $\mathcal{L}_x$ by
$$ \mathcal{L}_x(\Sigma) = n \int_{\Omega} \rho_x \ \! d\Omega,$$
where $\Omega$ is the inner region bounded by $\Sigma$. If $x$ is the origin, we write only $\mathcal{L}(\Sigma)$ to denote $\mathcal{L}_x(\Sigma)$, that is,
$$\mathcal{L}(\Sigma) = n \int_{\Omega} \rho \ \! d\Omega.$$

The following is the main result of this paper.
\begin{theorem} \label{main}
Let $\Sigma \subset \mathbb{S}^n$ be a closed, orientable and connected embedded hypersurface.
If $\Sigma$ is strictly convex, then
\begin{equation} \label{AF_inequality_sphere}
\int_{\Sigma} \rho_x H d\Sigma \geq (n-1)\omega_{n-1} \left( \frac{\mathcal{L}_x(\Sigma)}{\mathcal{K}(\Sigma)} \right) \left[ \left( \frac{|\Sigma|}{\omega_{n-1}} \right)^{\frac{n-2}{n-1}} - 
\left( \frac{|\Sigma|}{\omega_{n-1}} \right)^{\frac{n}{n-1}} \right],
\end{equation}
where $x = x(\Sigma)$ is the point associated to $\Sigma$ via the IMCF.
Morever, if the equality holds, then $\Sigma$ is a geodesic sphere centered at $x$.
\end{theorem}

If $x \in \mathbb{S}^n$, we denote the open hemisphere centered at $x$ (looking at the hemisphere as a geodesic ball) by $\mathbb{S}^n_+(x)$. If $x$ is the origin we write only $\mathbb{S}^n_+$.

As we will see later (Proposition \ref{prop_inequality_j_k}), it holds
$$\mathcal{L}_x(\Sigma) \leq \mathcal{K}(\Sigma),$$
with equality occurring if and only if $\Sigma$ is a geodesic sphere centered at $x$ and contained in $\mathbb{S}^n_+(x)$.
Thus, unless $\Sigma$ is a geodesic sphere, Conjecture \ref{conj} doesn't follow from Theorem \ref{main}.

A strictly convex hypersurface $\Sigma \subset \mathbb{S}^n$ is called {\it balanced} if $x(\Sigma)$, the point associated to $\Sigma$ via the IMCF, is the origin. Notice that, if $\Sigma$ is not balanced, then there exists an isometry $\Psi$ such that $\Psi (\Sigma)$ is balanced. Clearly, in order to prove Theorem \ref{main}, we can assume $\Sigma$ is balanced. If $\Sigma$ is balanced, then inequality (\ref{AF_inequality_sphere}) takes the following form:
\begin{equation} \label{AF_inequality_sphere_balanced}
\int_{\Sigma} \rho H d\Sigma \geq (n-1)\omega_{n-1} \left( \frac{\mathcal{L}(\Sigma)}{\mathcal{K}(\Sigma)} \right) \left[ \left( \frac{|\Sigma|}{\omega_{n-1}} \right)^{\frac{n-2}{n-1}} - 
\left( \frac{|\Sigma|}{\omega_{n-1}} \right)^{\frac{n}{n-1}} \right].
\end{equation}

Before passing to the next section, we would like to mention that the articles \cite{wei-xiong} and \cite{makowski-scheuer} also prove Alexandrov--Fenchel-type inequalities for hypersurfaces in the sphere.

\section{A monotone quantity along the IMCF}
Let $(N^{n-1},g_N)$ be a closed, orientable and connected Riemannian manifold. We consider the product manifold $M = N \times [0,r)$ endowed with the Riemannian metric
\begin{equation} \label{warped product structure}
 \overline{g} = dr^2 + \eta^2(r)g_N,
\end{equation}
where $\eta:[0,r) \to \mathbb{R}$ is a smooth function which is positive on $(0,r)$.

We consider a closed, orientable and connected embedded hypersurface $\Sigma \subset M$. As observed in \cite{brendle}, $M \setminus \Sigma$ has exactly two connected components. One of these components is contained in $N \times [0, r - \delta)$, for some $\delta > 0$. We call this component the {\it inner region} and denote it by $\Omega$. The other component is called the {\it outer region}. Notice that, when $M$ is an open hemisphere and $\Sigma$ is diffeomorphic to $\mathbb{S}^{n-1}$, then this definition of inner region coincides with the previously given one. 

We denote by $\xi$ the unit normal vector field along $\Sigma$ which points to the inner region. The Levi--Civita connections of $M$ and $\Sigma$ are denoted by $D$ and $\nabla$, respectively. We denote by $g$ and $b$, respectively, the metric and the second fundamental form of $\Sigma$. Thus, if $X$ and $Y$ are vector fields tangent to $\Sigma$, then
$$b(X,Y) = g(aX,Y),$$
where
$$aX = -D_X\xi$$
is the shape operator. 
As before, we denote by $\lambda = (\lambda_1, \ldots, \lambda_{n-1})$ the principal curvature vector of $\Sigma$ and by
$$H = \sigma_1(\lambda) = \mathrm{tr}_gb$$
its mean curvature.
We also consider the extrinsic scalar curvature of $\Sigma$, namely
\begin{equation} \label{extrinsic_scalar_curvature_warped}
K = \sigma_2(\lambda) = \sum_{i<j}\lambda_i\lambda_j = \frac{1}{2}\left(H^2 - |a|^2 \right).
\end{equation}
Notice that, by the Newton-MacLaurin inequality,
\begin{equation} \label{Newton_MacLaurin_warped}
2K \leq \frac{n-1}{n-2}H^2,
\end{equation}
with equality occurring at a given point if and only if $\Sigma$ is umbilical there.

Let the function $\rho: M \to \mathbb{R}$ be defined by
\begin{equation} \label{defn_rho_warped}
\rho = \eta^{\prime}(r),
\end{equation}
where prime means derivative with respect to $r$. Notice that, if $N = (\mathbb{S}^{n-1},h)$ and $\eta(r) = \sin(r)$, so that $(M, \overline{g})$ is (part of) the round unit sphere, than the function (\ref{defn_rho_warped}) coincides with (\ref{defn_rho_sphere}).

We assume throughout this section that $(M,\overline{g})$ has scalar curvature 
\begin{equation} \label{scalar curvature}
R_{\overline{g}} = n(n-1)
\end{equation}
and that $\rho$ satisfies the identity
\begin{equation} \label{static_warped}
\left( \Delta_M \rho \right)\overline{g} - D^2 \rho + \rho \mathrm{Ric}_M = 0.
\end{equation}
Taking the trace we find
\begin{equation} \label{traced_static_warped}
\Delta_M \rho = -n\rho.
\end{equation}

\begin{remark} The warped product structure (\ref{warped product structure}) and the identities (\ref{scalar curvature}) and (\ref{static_warped}) are satisfied, for example, by the de Sitter-Schwarzchild manifold.
\end{remark}

The {\it support function} $p$ is defined by
\begin{equation} \label{defn_p_2_warped}
p = \overline{g}(D\rho, \xi),
\end{equation}
where $\rho$ is the function given by (\ref{defn_rho_warped}). We define the quantities $\mathcal{J}$ and $\mathcal{I}$, respectively, by
$$\mathcal{J}(\Sigma) = \int_\Sigma p d\Sigma$$
and
$$\mathcal{I}(\Sigma) = \int_\Sigma \rho H d\Sigma.$$
Notice that (\ref{traced_static_warped}) implies
\begin{equation} \label{relation_J_and_L}
\mathcal{J}(\Sigma) = n\int_\Omega \rho \ \! d\Omega + C,
\end{equation}
where $C$ is a constant.

We consider an one-parameter family $X(t, \cdot): \Sigma \to M$, $t \in [0,\epsilon)$, of isometrically embedded hypersurfaces, satisfying $X(0,\cdot) = \Sigma$ and evolving according to
\begin{equation} \label{flow_warped}
\frac{\partial X}{\partial t} = F \xi,
\end{equation}
where $\xi$ is the unit normal to $\Sigma_t = X(t, \cdot)$ and $F$ is a general speed function. If no confusion arises, we sometimes write the evolving hypersurface simply by $\Sigma$.

\begin{proposition}
It holds
\begin{equation} \label{relation_laplacians}
\Delta_{\Sigma} \rho = -\rho \mathrm{Ric}_M (\xi, \xi) + Hp.
\end{equation}
\end{proposition}
\begin{proof}
If $X$ is tangent to $\Sigma$, on one hand we have
\begin{equation} \label{eq_1}
D_X \nabla \rho = \nabla_X \nabla \rho + g(aX,\nabla \rho)\xi.
\end{equation}
On the other hand,
\begin{align}
D_X \nabla \rho =& D_X \left( D \rho - g(D \rho, \xi)\xi \right)\nonumber \\
=& D_X D \rho - X g(D \rho, \xi) \xi - g(D \rho, \xi) D_X \xi. \label{eq_2}
\end{align}
Thus, if $Y$ is tangent to $\Sigma$, we find from (\ref{eq_1}) and (\ref{eq_2}) that
\begin{equation} \label{eq_3}
g( \nabla_X \nabla \rho, Y) = \overline{g}( D_X D \rho, Y) + \overline{g}(D\rho,\xi)g(aX,Y).
\end{equation}
From (\ref{eq_3}) we get
\begin{equation} \label{laplacians_and_mean_curvature}
\Delta_{\Sigma} \rho = \Delta_M \rho - D^2\rho(\xi,\xi) + \overline{g}(D\rho,\xi)H.
\end{equation}
Identity (\ref{relation_laplacians}) follows from (\ref{laplacians_and_mean_curvature}) and (\ref{static_warped}).
\end{proof}

\begin{proposition} \label{evolution_equations_warped}
Under the flow (\ref{flow_warped}), the following evolution equations hold:
\begin{itemize}
\item[] The unit normal evolves as
\begin{equation} \label{evolution_unit_normal_warped}
\frac{\partial \xi}{\partial t} = - \nabla F.
\end{equation}
\item[] The area element $d\Sigma$ evolves as
\begin{equation} \label{evolution_are_element_warped}
\frac{\partial}{\partial t} d\Sigma = -FHd\Sigma;
\end{equation}
in particular, the area $|\Sigma|$ evolves as
\begin{equation} \label{evolution_area_warped}
\frac{d}{d t}|\Sigma| = -\int_\Sigma FHd\Sigma.
\end{equation}
\item[] The mean curvature evolves as
\begin{equation} \label{evolution_mean_curvature_warped}
\frac{\partial H}{\partial t} = \Delta_\Sigma F + \left( |a|^2 + \mathrm{Ric}_M(\xi,\xi) \right)F.
\end{equation}
\item[] The function $\rho$ evolves as
\begin{equation} \label{evolution_rho_warped}
\frac{\partial \rho}{\partial t} = pF.
\end{equation}
\item[] The quantity $\mathcal{J}$ evolves as
\begin{equation} \label{evolution_J_warped}
\frac{d\mathcal{J}}{d t} = -n \int_{\Sigma} F  \rho  d\Sigma.
\end{equation}
\item[] The quantity $\mathcal{I}$ evolves as
\begin{equation} \label{evolution_I_warped}
\frac{d \mathcal{I}}{d t} 
= 2\int_\Sigma pHF d\Sigma - 2\int_\Sigma \rho K F d\Sigma.
\end{equation}
\end{itemize}
\end{proposition}
\begin{proof}
Formulas (\ref{evolution_unit_normal_warped}) -- (\ref{evolution_mean_curvature_warped}) are well known (see, for example, \cite{zhu}).
By (\ref{flow_warped}) and (\ref{defn_p_2_warped}) we have 
\begin{align*}
\frac{\partial \rho}{\partial t} &= \overline{g} (D \rho, \frac{\partial X}{\partial t} ) \\
&= \overline{g} ( D \rho, F \xi ) \\
&= Fp,
\end{align*}
which proves (\ref{evolution_rho_warped}). Equation (\ref{evolution_J_warped}) follows from (\ref{relation_J_and_L}), (\ref{flow_warped}) and the coarea formula. It remains to show (\ref{evolution_I_warped}). Using (\ref{evolution_rho_warped}), (\ref{evolution_mean_curvature_warped}) and (\ref{evolution_are_element_warped}), we have
\begin{align*}
\frac{\partial\mathcal{I}}{\partial t} 
&= \int_{\Sigma} \frac{\partial \rho}{\partial t} H d\Sigma + \int_{\Sigma} \rho \frac{\partial H}{\partial t} d\Sigma + \int_{\Sigma} \rho H \frac{\partial}{\partial t} d\Sigma \\
&= \int_{\Sigma} pFH d\Sigma + \int_{\Sigma} \rho \left( \Delta_\Sigma F + \left( |a|^2 + \mathrm{Ric}_M(\xi,\xi) \right)F \right) d\Sigma - \int_{\Sigma} \rho F H^2 d\Sigma \\
&= \int_{\Sigma} pFH d\Sigma + \int_{\Sigma}F \Delta_\Sigma \rho d\Sigma + \int_{\Sigma} \rho \left( |a|^2 + \mathrm{Ric}(\xi,\xi) \right)F d\Sigma - \int_{\Sigma} \rho F H^2 d\Sigma,
\end{align*}
and the result follows, after some cancellations, from (\ref{relation_laplacians}) and (\ref{extrinsic_scalar_curvature_warped}).
\end{proof}

From now on we restrict ourselves to the inverse mean curvature flow
\begin{equation} \label{IMCF_warped}
\frac{\partial X}{\partial t} = -\frac{\xi}{H},
\end{equation}
so that $F = -1/H$.

Let $\mathcal{A}(\Sigma) = |\Sigma|/\omega_{n-1}.$ Notice that, by (\ref{evolution_area_warped}),
\begin{equation} \label{variation_mathcal_A_warped}
\frac{d \mathcal{A}}{dt} = \mathcal{A}.
\end{equation}

\begin{proposition} \label{prop_monotone_quantity}
Suppose 
\begin{equation} \label{A_leq_1_warped}
\mathcal{A} \leq 1
\end{equation}
and
\begin{equation} \label{Brendle_with_alfa}
(n-1) \int_\Sigma \frac{\rho}{H} d\Sigma \geq \mathcal{J}(\Sigma) + \alpha,
\end{equation}
for some constant $\alpha$.
Then, along the flow (\ref{IMCF_warped}), the quantity $\mathcal{Q}$ defined by
\begin{equation} \label{defn_Q_warped}
\mathcal{Q} = \mathcal{A}^{-\frac{n-2}{n-1}} \left[ \mathcal{I} - (n-1) \left(\mathcal{J} + \beta \right) \left( \mathcal{A}^{-\frac{2}{n-1}} -1 \right) + \gamma \mathcal{A}^{\frac{-2}{n-1}} \right]
\end{equation}
satisfies
\begin{equation} \label{Q_is_decreasing}
\frac{d \mathcal{Q}}{dt} \leq 0,
\end{equation}
where
\begin{equation} \label{beta_e_gama}
\beta =  \left( \frac{n-1}{n-2} \right) \alpha \ \ \  \text{and} \ \ \ \gamma = \frac{2(n-1)^2}{n(n-2)} \alpha.
\end{equation}
\end{proposition}
\begin{proof}
By (\ref{evolution_I_warped}) and (\ref{Newton_MacLaurin_warped}) we have that
\begin{equation} \label{inequality_variation_I_warped}
\frac{d \mathcal{I}}{d t}
\leq \frac{n-2}{n-1} \mathcal{I} - 2\mathcal{J}.
\end{equation}
Also, by (\ref{evolution_J_warped}) and (\ref{Brendle_with_alfa}), it holds
\begin{equation} \label{inequality_variation_J_warped}
\frac{d \mathcal{J}}{dt} \geq \left( \frac{n}{n-1} \right) \mathcal{J} + \alpha.
\end{equation}
Using (\ref{variation_mathcal_A_warped}), (\ref{A_leq_1_warped}) and (\ref{beta_e_gama}) -- (\ref{inequality_variation_J_warped}) we find that
\begin{align*}
\frac{d}{dt} \left(\mathcal{A}^{\frac{n-2}{n-1}} \mathcal{Q} \right)
& =  \frac{d \mathcal{I}}{dt} - (n-1) \frac{d \mathcal{J}}{dt} \left( \mathcal{A}^{-\frac{2}{n-1}} - 1 \right) + 2 \left( \mathcal{J} + \beta - \frac{1}{n-1} \right) {\mathcal{A}}^{-\frac{2}{n-1}}  \\
& \leq \left( \frac{n-2}{n-1} \right) \mathcal{I} - 2 \mathcal{J} - (n-1) \left(  \frac{n}{n-1}  \mathcal{J} + \alpha \right) \left( \mathcal{A}^{-\frac{2}{n-1}} - 1 \right) \\ 
& \ \ \ \ + 2 \left( \mathcal{J} + \beta - \frac{1}{n-1} \right) {\mathcal{A}}^{-\frac{2}{n-1}} \\
 & = \left( \frac{n-2}{n-1} \right)\left(\mathcal{A}^{\frac{n-2}{n-1}} \mathcal{Q} \right),
\end{align*}
which proves (\ref{Q_is_decreasing}).
\end{proof}

\section{Proof of Theorem \ref{main}}
Let $\Sigma \subset \mathbb{S}^n$ be a strictly convex hypersurface.
Without loss of generality, we can assume that $\Sigma$ is balanced.

It is proved in \cite{makowski-scheuer} that the IMCF is smooth on an interval $[0,T^\ast)$, with $\Sigma_t$ converging to an equator, as $t \to T^\ast$, and with
\begin{equation} \label{integral_of_H}
\lim_{t \to T^\ast} \int_{\Sigma} H d\Sigma = 0.
\end{equation}
Since $\Sigma$ is balanced, the restriction of $\rho$ to $\Sigma_t$ satisfies
$$ 0 < \rho < 1,$$
for each $t \in [0,T^\ast)$. Thus, as $\Sigma_t$ remains strictly convex along the flow (\cite{makowski-scheuer}), we have
$$ 0 < \rho H < H,$$
for each $t \in [0,T^\ast)$. Hence, by (\ref{integral_of_H}) and the squeeze theorem,
\begin{equation} \label{limite_int_rho_H}
\lim_{t \to T^\ast} \int_{\Sigma} \rho H d\Sigma = 0.
\end{equation}

Using again that $\Sigma_t$ converges to an equator, it follows that
\begin{equation} \label{lim_A}
\lim_{t \to T^\ast} \mathcal{A}(t) = 1.
\end{equation}
Hence, since $\mathcal{A}$ is increasing along the flow, we find that
\begin{equation} \label{A_geq_1}
\mathcal{A} \leq 1.
\end{equation}
Moreover, it is proved in \cite{brendle} that the inequality
\begin{equation}
(n-1) \int_\Sigma \frac{\rho}{H} d\Sigma \geq \mathcal{J}(\Sigma)
\end{equation}
holds (see also \cite{qiu-xia}, \cite{miao-kwong}, \cite{li-xia}, and \cite{wang-wang}). Thus, we may invoke Proposition \ref{prop_monotone_quantity} and conclude that the quantity
\begin{equation} \label{quantity_Q}
\mathcal{Q} = \mathcal{A}^{-\frac{n-2}{n-1}} \left[ \mathcal{I} - (n-1) \mathcal{J} \left( \mathcal{A}^{-\frac{2}{n-1}} -1 \right) \right]
\end{equation}
satisfies
$$ \frac{d \mathcal{Q}}{dt} \leq 0 $$
along the IMCF.
Thus, by the fundamental theorem of calculus,
\begin{equation} \label{inequality_TFC}
\mathcal{Q}(t) - \mathcal{Q}(0) \leq 0,
\end{equation}
for all $t \in [0,T^\ast)$.

By (\ref{limite_int_rho_H}) and (\ref{lim_A}) we have
$$ \lim_{t \to T^\ast} \mathcal{Q}(t) = 0.$$
This, together with (\ref{inequality_TFC}), implies 
$$ \mathcal{Q}(0) \geq 0,$$
which is equivalent to inequality (\ref{AF_inequality_sphere_balanced}).

It remains to show the rigidity statement. If $\Sigma$ is a geodesic sphere centered at the origin, then a straightforward computation gives that the equality holds in (\ref{AF_inequality_sphere_balanced}).
If the equality holds in (\ref{AF_inequality_sphere_balanced}), then
$$ \frac{d \mathcal{Q}}{dt} = 0,$$
for each $t \in [0,T^\ast)$. Thus, the equality also holds in (\ref{inequality_variation_I_warped}), for each $t \in [0,T^\ast)$. But, if the equality holds in (\ref{inequality_variation_I_warped}), for each $t \in [0,T^\ast)$, then it also holds in (\ref{Newton_MacLaurin_warped}), for each $t \in [0,T^\ast)$. As a consequence, $\Sigma_0$ is totally umbilical, and hence, it is a geodesic sphere. Since $\Sigma_0$ is balanced, it is a geodesic sphere centered at the origin.

\section{Secondary results}
In this section we prove the remaining two results mentioned in the Introduction.

\begin{proposition} \label{prop_inequality_j_k}
Let $x \in \mathbb{S}^n$ and let $\Sigma \subset \mathbb{S}^n$ be a closed, orientable and connected embedded hypersurface.
Then it holds
$$ \mathcal{L}_x (\Sigma) \leq \mathcal{K}(\Sigma),$$
with the equality holding if and only if $\Sigma$ is a geodesic sphere centered at $x$ and contained in $\mathbb{S}^n_+(x)$.
\end{proposition}

\begin{proof}
We can suppose, without loss of generality, that $x$ is the origin. We then need to show that
$$ \mathcal{L}(\Sigma) \leq \mathcal{K}(\Sigma),$$
with equality holding if and only if $\Sigma$ is a geodesic sphere centered at the origin and contained in $\mathbb{S}^n_+$.

Let $\hat{\Sigma}$ be the geodesic sphere centered at the origin, with $\hat{\Sigma}$ contained in $\mathbb{S}^n_+$ and such that 
$$ \mathcal{K}(\Sigma) = \mathcal{K}(\hat{\Sigma}).$$
A straightforward computation shows that
\begin{equation} \label{j=k}
\mathcal{L}(\hat{\Sigma}) = \mathcal{K}(\hat{\Sigma}).
\end{equation}
Denote by $\Omega$ and $\hat{\Omega}$, respectively, the inner regions of $\Sigma$ and $\hat{\Sigma}$. By the isoperimetric inequality, we have
\begin{equation} \label{inequality_volume_2}
\int_{\hat{\Omega} \setminus \Omega} 1 dV \leq \int_{\Omega \setminus \hat{\Omega}} 1 dV,
\end{equation}
with the equality holding if and only if $\Sigma$ is a geodesic sphere.

It follows from (\ref{inequality_volume_2}) and the fact that $\rho$ is decreasing with respect to $r$ that
\begin{equation} \label{inequality setminus}
\int_{\hat{\Omega} \setminus \Omega} \rho \ \! dV \leq \int_{\Omega \setminus \hat{\Omega}} \rho \ \! dV,
\end{equation}
with the equality holding if and only if $\Sigma=\hat{\Sigma}$.

Using (\ref{j=k}) and (\ref{inequality setminus}) we get
\begin{align*}
\mathcal{L}(\Sigma)
&= n \int_{\Omega \setminus \hat{\Omega}} \rho \ \! dV + n \int_{\Omega \cap \hat{\Omega}} \rho \ \! dV \\
&\leq n \int_{\hat{\Omega} \setminus \Omega} \rho \ \! dV + n \int_{\Omega \cap \hat{\Omega}} \rho \ \! dV \\
& = \mathcal{L}(\hat{\Sigma})  \\
& = \mathcal{K}(\Sigma),
\end{align*}
with the equality holding if and only if $\Sigma = \hat{\Sigma}$, that is, if and only if $\Sigma$ is a geodesic sphere centered at the origin and contained in $\mathbb{S}^n_+$.
\end{proof}

\begin{proposition} \label{result_1}
If $\Sigma$ is a geodesic sphere not centered at the origin, then
$$ \int_{\Sigma} \rho H d\Sigma < (n-1)\omega_{n-1} \left[ \left( \frac{|\Sigma|}{\omega_{n-1}} \right)^{\frac{n-2}{n-1}} - 
\left( \frac{|\Sigma|}{\omega_{n-1}} \right)^{\frac{n}{n-1}} \right].$$
\end{proposition}
\begin{proof}
Denote by $\hat{\Sigma}$ the geodesic sphere centered at the origin, with $\hat{\Sigma}$ contained in $\mathbb{S}^n_+$ and such that
\begin{equation} \label{same_area}
| \hat{\Sigma} | = | \Sigma |.
\end{equation}
 A straightforward computation gives
$$
\mathcal{I}(\hat{\Sigma}) =
(n-1)\omega_{n-1} \left[ \left( \frac{|\hat{\Sigma}|}{\omega_{n-1}} \right)^{\frac{n-2}{n-1}} - 
\left( \frac{|\hat{\Sigma}|}{\omega_{n-1}} \right)^{\frac{n}{n-1}} \right].
$$
Thus, by (\ref{same_area}), it is enough to show that 
$$\mathcal{I}(\Sigma)  < \mathcal{I}(\hat{\Sigma}).$$

Noticed that the vector field $D \rho$ is conformal and
$$ \nabla \rho = \left( D \rho \right)^{\top}.$$
Thus, it follows from formula (8.4) of \cite{alias-delira-malacarne} that
\begin{equation} \label{second_Minkowski_warped}
\mathrm{div}(G \nabla \rho) = -(n-2)\rho H + 2pK.
\end{equation}
Integrating (\ref{second_Minkowski_warped}) and using that $K$ is constant we find
\begin{equation} \label{relation_I_J_geodesic_sphere}
(n-2) \mathcal{I}(\Sigma) = 2K \mathcal{L}(\Sigma),
\end{equation}
where we have used that
$$ \mathcal{J}(\Sigma) = \mathcal{L}(\Sigma).$$
In a similar way,
\begin{equation} \label{relation_I_J_geodesic_sphere_hat}
(n-2) \mathcal{I}(\hat{\Sigma}) = 2\hat{K} \mathcal{L}(\hat{\Sigma}),
\end{equation}
where $\hat{K}$ is the extrinsic scalar curvature of $\hat{\Sigma}$.

Therefore, since $K = \hat{K}$, the result follows from Proposition \ref{prop_inequality_j_k} together with (\ref{relation_I_J_geodesic_sphere}) and (\ref{relation_I_J_geodesic_sphere_hat}).

\end{proof}


\begin{thebibliography}{99}

\bibitem{alexandrov1}
A.~D. Alexandrov.
\newblock {Zur Theorie der gemischten Volumina von konvexen K\"orpern. II. Neue
  Ungleichungen zwischen den gemischten Volumina und ihre Anwendungen.}
\newblock {\em Rec. Math. (Moscou) [Mat. Sbornik] N.S.}, 2:1205--1238, 1937.

\bibitem{alexandrov2}
A.~D. Alexandrov.
\newblock {Zur Theorie der gemischten Volumina von konvexen K\"orpern III. Die
  Erweiterung zweier Lehrs\"atze Minkowskis \"uber die konvexen Polyeder auf
  die beliebigen konvexen K\"orper.}
\newblock {\em Rec. Math. (Moscou) [Mat. Sbornik] N.S.}, 3:27--46, 1938.

\bibitem{alias-delira-malacarne}
L.~J. Al{\'{\i}}as, J.~H.~S. de~Lira, and J.~M. Malacarne.
\newblock Constant higher-order mean curvature hypersurfaces in {R}iemannian
  spaces.
\newblock {\em J. Inst. Math. Jussieu}, 5(4):527--562, 2006.

\bibitem{brendle}
S.~Brendle.
\newblock Constant mean curvature surfaces in warped product manifolds.
\newblock {\em Publ. Math. Inst. Hautes \'Etudes Sci.}, 117:247--269, 2013.

\bibitem{BHW}
S.~Brendle, P.-K. Hung, and M.-T. Wang.
\newblock A {M}inkowski inequality for hypersurfaces in the anti-de
  {S}itter-{S}chwarzschild manifold.
\newblock {\em Comm. Pure Appl. Math.}, 69(1):124--144, 2016.

\bibitem{DGS}
M.~Dahl, R.~Gicquaud, and A.~Sakovich.
\newblock Penrose type inequalities for asymptotically hyperbolic graphs.
\newblock {\em Ann. Henri Poincar\'e}, 14(5):1135--1168, 2013.

\bibitem{dLG1}
L.~L. de~Lima and F.~Gir{\~a}o.
\newblock The {ADM} mass of asymptotically flat hypersurfaces.
\newblock {\em Trans. Amer. Math. Soc.}, 367(9):6247--6266, 2015.

\bibitem{dLG2}
L.~L. de~Lima and F.~Gir{\~a}o.
\newblock An {A}lexandrov-{F}enchel-type inequality in hyperbolic space with an
  application to a {P}enrose inequality.
\newblock {\em Ann. Henri Poincar\'e}, 17(4):979--1002, 2016.

\bibitem{docarmo-warner}
M.~P. do~Carmo and F.~W. Warner.
\newblock Rigidity and convexity of hypersurfaces in spheres.
\newblock {\em J. Differential Geometry}, 4:133--144, 1970.

\bibitem{GWW}
Y.~Ge, G.~Wang, and J.~Wu.
\newblock A new mass for asymptotically flat manifolds.
\newblock {\em Adv. Math.}, 266:84--119, 2014.

\bibitem{GWW2}
Y.~Ge, G.~Wang, and J.~Wu.
\newblock The {GBC} mass for asymptotically hyperbolic manifolds.
\newblock {\em Math. Z.}, 281(1-2):257--297, 2015.

\bibitem{gerhardt}
C.~Gerhardt.
\newblock Curvature flows in the sphere.
\newblock {\em J. Differential Geom.}, 100(2):301--347, 2015.

\bibitem{girao-mota}
F.~{Gir{\~a}o} and A.~{Mota}.
\newblock {The Gauss-Bonnet-Chern mass of higher codimension graphical
  manifolds}.
\newblock {\em ArXiv e-prints}, Sept. 2015.

\bibitem{guan-li}
P.~Guan and J.~Li.
\newblock The quermassintegral inequalities for {$k$}-convex starshaped
  domains.
\newblock {\em Adv. Math.}, 221(5):1725--1732, 2009.

\bibitem{miao-kwong}
K.-K. {Kwong} and P.~{Miao}.
\newblock {A functional inequality on the boundary of static manifolds}.
\newblock {\em ArXiv e-prints}, Jan. 2016.

\bibitem{lam}
M.-K.~G. Lam.
\newblock {\em The {G}raph {C}ases of the {R}iemannian {P}ositive {M}ass and
  {P}enrose {I}nequalities in {A}ll {D}imensions}.
\newblock ProQuest LLC, Ann Arbor, MI, 2011.
\newblock Thesis (Ph.D.)--Duke University.

\bibitem{li-wei-xiong}
H.~Li, Y.~Wei, and C.~Xiong.
\newblock The {G}auss-{B}onnet-{C}hern mass for graphic manifolds.
\newblock {\em Ann. Global Anal. Geom.}, 45(4):251--266, 2014.

\bibitem{li-xia}
J.~{Li} and C.~{Xia}.
\newblock {An integral formula and its applications on sub-static manifolds}.
\newblock {\em ArXiv e-prints}, Mar. 2016.

\bibitem{makowski-scheuer}
M.~{Makowski} and J.~{Scheuer}.
\newblock {Rigidity results, inverse curvature flows and Alexandrov-Fenchel
  type inequalities in the sphere}.
\newblock {\em ArXiv e-prints}, July 2013.

\bibitem{mirandola-vitorio}
H.~Mirandola and F.~Vit{\'o}rio.
\newblock The positive mass theorem and {P}enrose inequality for graphical
  manifolds.
\newblock {\em Comm. Anal. Geom.}, 23(2):273--292, 2015.

\bibitem{qiu-xia}
G.~Qiu and C.~Xia.
\newblock A generalization of {R}eilly's formula and its applications to a new
  {H}eintze-{K}archer type inequality.
\newblock {\em Int. Math. Res. Not. IMRN}, (17):7608--7619, 2015.

\bibitem{wang-wang}
X.~{Wang} and Y.-K. {Wang}.
\newblock {Brendle's inequality on static manifolds}.
\newblock {\em ArXiv e-prints}, Mar. 2016.

\bibitem{wei-xiong}
Y.~Wei and C.~Xiong.
\newblock Inequalities of {A}lexandrov-{F}enchel type for convex hypersurfaces
  in hyperbolic space and in the sphere.
\newblock {\em Pacific J. Math.}, 277(1):219--239, 2015.

\bibitem{zhu}
X.-P. Zhu.
\newblock {\em Lectures on mean curvature flows}, volume~32 of {\em AMS/IP
  Studies in Advanced Mathematics}.
\newblock American Mathematical Society, Providence, RI; International Press,
  Somerville, MA, 2002.
  
\end{thebibliography}


\begin{multicols}{2}
\noindent
Frederico Gir\~ao \\
Universidade Federal do Cear\'a \\
{\tt fred@mat.ufc.br} \\
Neilha M. Pinheiro \\
Universidade Federal do Cear\'a \\
{\tt neilhamat@gmail.com}
\end{multicols}
\end{document}